\documentclass[a4paper]{article}

\usepackage{marginnote}
\makeatletter
\newcommand{\edit}{\@ifstar\edit@star\edit@nostar}
\newcommand{\edit@star}[1]{{\color{blue}#1}}
\newcommand{\edit@nostar}[2][\@]{%
	\marginnote{%
	\if#1\@
		\"Andringar
	\else
		\"Andring: #1
	\fi
	}{\color{blue}#2}}
\newcommand{\newedit}{\@ifstar\newedit@star\newedit@nostar}
\newcommand{\newedit@star}[1]{{\color{red}#1}}
\newcommand{\newedit@nostar}[2][\@]{%
	\marginnote{%
	\if#1\@
		\"Andringar
	\else
		\"Andring: #1
	\fi
	}{\color{red}#2}}
\makeatother
\newcommand{\comment}[1]{\bgroup\reversemarginpar\marginnote{\null\hfill Kommentar\hfill\null}\emph{#1}\egroup}

\usepackage[T1]{fontenc}
\usepackage[utf8]{inputenc}
\usepackage[pass]{geometry}
\usepackage{tikz}
	\usetikzlibrary{positioning}
\usepackage{amsmath}
\usepackage{amsfonts}
\usepackage{amssymb}
\usepackage{mathtools}
\usepackage{amsthm}
\usepackage{etoolbox}
\usepackage{graphicx}
\usepackage{caption}
\RequirePackage{hyperref}
	\hypersetup{linktoc=all,hidelinks}
\usepackage[capitalize,nameinlink]{cleveref}
\usepackage[noadjust]{cite}

\theoremstyle{definition}
\newtheorem{theorem}{Theorem}[section]
\newtheorem{lemma}[theorem]{Lemma}
\newtheorem{proposition}[theorem]{Proposition}

\newtheorem{problem}[theorem]{Problem}
\newtheorem{corollary}[theorem]{Corollary}

\newtheorem{example}{Example}[section]
\newtheorem*{definition}{Definition}
\theoremstyle{remark}
\newtheorem{remark}{Remark}[section]
\crefname{claim}{Claim}{Claims}
\crefname{case}{Case}{Cases}
\crefname{subcase}{Case}{Cases}
\crefname{subsubcase}{Case}{Cases}
\crefname{conjecture}{Conjecture}{Conjectures}
\AfterEndEnvironment{case}{\noindent\ignorespaces}
\AfterEndEnvironment{subcase}{\noindent\ignorespaces}
\AfterEndEnvironment{subsubcase}{\noindent\ignorespaces}
\newcounter{parenum}[theorem]

\let\eqref\labelcref
\crefname{equation}{}{}

\crefname{enumi}{}{}



\newlength{\revdirraise}
\newlength{\revdirextraraise}

\let\overrightarrow\dir
\let\overleftarrow\revdir

\DeclarePairedDelimiterX{\defset}[2]{\lbrace}{\rbrace}
	{\,#1:#2\,}

\AtBeginDocument{\frenchspacing}

\hypersetup{%
	pdftitle={On the edge chromatic transformation index of graphs},
	pdfauthor={Armen S. Asratian, \  \  Carl Johan Casselgren},
	pdfkeywords={Edge coloring, Interchange, Kempe equivalence}}

\begin{document}

\title{The edge chromatic transformation  index of graphs}

\author{Armen S. Asratian%
	\footnote{Department of Mathematics, Linköping University, email: armen.asratian@liu.se},
	Carl Johan Casselgren %
	\footnote{Department of Mathematics, Linköping University, email: carl.johan.casselgren@liu.se}}
\date{}
\maketitle

\begin{abstract}
Given a graph or multigraph  $G$,  let $\chi'_{trans}(G)$ denote the minimum integer $n$
such that any proper $\chi'(G)$--edge coloring  of $G$ can be transformed into any other proper $\chi'(G)$--edge coloring of $G$
by a series of transformations such that each of the intermediate colorings is a proper $\chi'(G)$--edge coloring of $G$ and 
each of the transformations involves
at most $n$ color classes of the previous coloring.  
We call  $\chi'_{trans}(G)$ the {\it  edge chromatic transformation  index of $G$}.

In this paper we show that if $G$ is a  
graph with maximum degree at least $4$, where every block is either  a bipartite graph,
a series-parallel graph,    a chordless graph, a wheel graph 
or a planar graph
of girth at least $7$, then $\chi'_{trans}(G)\leq 4$. This bound is sharp for series-parallel and wheel 
graphs. We  also show that 
$\chi'_{trans}(G)\leq 8$ for  all  planar graphs $G$, $\chi'_{trans}(G)\leq 5$ if $G$ is a Halin graph and  $\chi'_{trans}(G)=2$ if  $G$ is a regular bipartite planar  multigraph. Finally,  we consider the analogous problem
for vertex colorings, and
show that for any $k\geq 3$ there is an infinite class $\cal G$$(k)$ of   graphs   with chromatic number 
$k$ such that for every $G\in \cal G$$(k)$ any two    proper $k$-vertex colorings of $G$ 
 can be transformed to each other   
only by a   transformation, involving all $k$ color classes. 

\end{abstract}

\medskip
\bgroup
\par
\egroup

\section{Introduction}

Edge coloring problems appear in many places with seemingly no or little connections to graph
coloring \cite{AsrCassPetros,AsrWerra,AsrWerraDurand,coffman,Januario,Marx}. For example, many problems on school 
 timetables can be formulated in terms of edge colorings of bipartite graphs and multigraphs \cite{AsrDenHag,Werra}.
Some optimization problems formulated in terms of  edge colorings are NP-hard and therefore   handled with heuristic algorithms. 
Usually the underlying search problem is solved and the resulting feasible solution (a proper coloring) is used as a starting point 
to obtain a better solution.
Therefore we need some types of transformations which transform one proper coloring (timetable) to any other proper coloring (timetable). Moreover it is convenient to change colorings partly using as few color classes as possible.

In this paper,  graphs are finite, undirected, without loops and multiple edges.
In multigraphs multiple edges are allowed but not loops.
A 
{\em $t$-edge coloring} or simply {\em $t$-coloring} 
of a graph or multigraph $G=(V(G),E(G))$ is a
mapping
$f: E(G)\longrightarrow  \{1,\dots,t\}$.
If $e\in E(G)$ and $f(e)=j$ then we say that the edge $e$ is colored $j$. 
The set of edges of color $j$, denoted  by $M(f,j)$, 
is called a {\em color class}, $j=1,\dots,t$.  
A $t$-coloring of $G$ is
called
{\em proper }
if no adjacent edges receive the same color. 
The minimum number $t$ for which there exists a  proper $t$-coloring of $G$ is called 
the {\em chromatic index} of  $G$ and is denoted by $\chi'(G)$.

We shall say that two distinct  $t$-colorings $f$ and $g$ {\em differ} by  $n$ color classes  
if there is a set of colors $S$ of size $n$, such that 
 $M(f,j)\not=M(g,j)$ for each $j\in S$, but $M(f,j)=M(g,j)$ for each $j\notin S$.

\begin{definition}
Let $f$ and $g$ be two proper  $t$-colorings of a graph $G$. We will say that $f$  is obtained from $g$ by a series of 
{\it $n$-transformations}, and $g$ and $f$ are {\it $n$-equivalent},
 if there is a sequence of proper $t$-colorings $f_0,f_1,\dots,f_k$ of $G$ such that 
 $k\geq 1$, $f_0=f, f_k=g$ and $f_i$ differs from $f_{i-1}$  by  at most $n$ color classes,  $i=1,\dots,k$.
A $2$-transformation is called a {\it Kempe change} or {\it interchange}, and $2$-equivalence of colorings
is known as {\it Kempe-equivalence}. 
\end{definition}

Vizing \cite{vizing64} proved that for every   graph $G$,
$\chi'(G)\le \Delta(G)+1$, where $\Delta(G)$ denotes the maximum   degree of the vertices of $G$.
This  implies that for any   graph $G$, either $\chi'(G)=\Delta(G)$ or $\chi'(G)=\Delta(G)+1$. In the former case
	$G$ is said to be \emph{Class $1$}, and in the latter $G$ is {\em Class $2$}. 
	
Investigations of transformations of proper edge colorings mostly concern two problems.
The first of them is due to  	Vizing \cite{vizing64,Vizing2} who showed that any    
proper edge coloring of a  graph $G$ can be transformed to a proper  
$(\Delta(G)+1)$-coloring of $G$
by using Kempe changes only, and posed the following problem:

\begin{problem}
\label{prob:vizing}
	Is it true that any Class 1 graph $G$ satisfies the following property:
	every proper $t$-coloring of $G$,  $t\geq \Delta(G)+1$, can be transformed to a proper $\Delta(G)$-coloring  
	of $G$ by
	a sequence of Kempe changes? 
\end{problem}

The second problem was posed by Mohar \cite{mohar07}.

\begin{problem}
\label{prob:mohar}
 Is it true that all proper $(\chi'(G)+1)$-colorings 
 of a   graph $G$ are Kempe-equivalent?
\end{problem}

Asratian and Casselgren showed  that in fact Problems \ref{prob:vizing} and \ref{prob:mohar} 
 have the same answer (see \cite{asratian16}, Theorem 1.1).

It was successively  proved that the answers to  Problems \ref{prob:vizing} and \ref{prob:mohar} 
are positive if either  
$\Delta(G)=3$ (McDonald,  Mohar and  Scheide \cite{mohar12}), $\Delta(G)=4$ (Asratian and Casselgren \cite{asratian16}),
$\Delta(G)\geq 9$ and $G$ is a planar graph (Cranston \cite{cranston}), 
$G$ is  triangle-free graph (Bonamy et al \cite{bonamy23}).
In 2023 Narboni \cite{narboni} published a preprint with a proof that the answers to 
Problems \ref{prob:vizing} and \ref{prob:mohar} are positive.

A few results are known on Kempe-equivalence
 of proper $\chi'(G)$-colorings 
of a  graph $G$:
It was proved that all proper $\chi'(G)$-colorings of 
a  graph $G$ are Kempe-equivalent if $G$ is  a
 planar graph with $\Delta(G)\geq 15$ \cite{cranston},   or a 3-regular planar bipartite multigraph \cite{belcastro14}.
It is  also known that there are infinite classes of graphs $G$ whose proper $\chi'(G)$-colorings are not Kempe-equivalent 
\cite{belcastro14,mohar07}. 

Therefore,  it is natural  to introduce  the following  parameter:\smallskip

 Given a graph or multigraph $G$, let 
 $\chi'_{trans}(G)$  denote the minimum integer $n\geq 2$
such that any proper $\chi'(G)$--coloring  of  $G$ can be transformed into any other proper $\chi'(G)$--coloring of $G$
by a series of $n$-transformations such that each intermediate coloring is a proper $\chi'(G)$--edge coloring. 
We call $\chi'_{trans}(G)$ the {\it edge chromatic transformation index} of $G$.

The following two problems arise naturally:

\begin{problem}
\label{prob:trans1}
Given a graph or multigraph $G$, find or estimate $\chi'_{trans}(G)$.
\end{problem}

\begin{problem}
\label{prob:trans}
Is there  an absolute constant $C$ such that $\chi'_{trans}(G)\leq C$ for every   graph  or multigraph $G$?
\end{problem}

Asratian and Mirumian \cite{asratian91} proved (in other terminology) that $\chi'_{trans}(G)\leq 3$ for any bipartite multigraph $G$.
(A shorter proof was  suggested by Asratian \cite{asratian09}).

In the present paper we show that the answer to Problem \ref{prob:trans} is positive 
if and only if such a constant exists for 
all  Class 1 graphs (see Proposition \ref{prop:Class1}). 
Hence   we shall investigate Problem \ref{prob:trans1} and Problem \ref{prob:trans}  only for Class 1 graphs. 

To state our results we need to introduce some notions  (see also Section 2).
\smallskip

Given an integer $q\geq 2$, let  $\cal A$$(q)$ denote the set of  graphs 
 where for each $G\in \cal A$$(q)$  every subgraph $H$ of $G$ with $\Delta(H)\leq q$ has a proper $q$-coloring. 
Note that 
the set $\cal A$$(2)$ is in fact the set of all bipartite graphs.

Our main result is the following:

\begin{theorem}
\label{main:th}
Let $q$ be an integer, $q\geq 3$, and $G$ be a Class 1 graph with $\Delta(G)\geq 5$. 
If every  block of $G$ is either  a bipartite graph,
or a $(q+1)$-degenerate graph\footnote{
For a positive integer $k$,
a graph $G$ is said to be $k$-degenerate  if for each
 subgraph $H$ of $G$, $H$ contains a vertex $x$ with $d_H(x)\leq k$.}
 from the set $\cal A$$(q)$, then $\chi'_{trans}(G)\leq q+1$.
\end{theorem}

Using this result, we show that if $G$ is a graph   with maximum degree at least $4$ 
where every  block is either 
a bipartite graph, a series-parallel graph, 
 a wheel graph, 
 a chordless graph, or a planar graph of girth at least 7, then $\chi'_{trans}(G)\leq 4$.
Note that this bound is sharp  for series-parallel and wheel graphs. We also show that
\begin{itemize}

	\item  $\chi'_{trans}(G)\leq 8$ if $G$ is an arbitrary planar graph, 
	
	\item $\chi'_{trans}(G)\leq 5$ if $G$ is a Halin graph, 
 
 	\item $\chi'_{trans}(G)=2$ if $G$ is a planar regular bipartite  multigraph. 
 		
\end{itemize}

 The latter statement
generalizes  a result of Belcastro and Haas \cite{belcastro14} for planar $3$-regular bipartite  multigraphs.

The answer to Problem \ref{prob:trans}, in general, remains  open. However in Section 5 we show that   the answer to a similar problem for vertex colorings is negative.
More precisely,  we show that for any $k\geq 3$ there exists an infinite class $\cal G$$(k)$ of  regular graphs   with chromatic number 
$\chi (G)=k$ such that for every $G\in \cal G$$(k)$ any two    proper $\chi(G)$-colorings of $G$ 
 can be transformed to each other   
only by a  "global" transformation, involving all $\chi(G)$ color classes.

The proofs in the present  paper are based on the method  suggested by Asratian in \cite{asratian09}.
In Section 2 we give some preliminaries, in Section 3 we prove a series of lemmas, and Section 4 contains the proofs of our main results.

 \section{Definitions and preliminary results}
 
A {\it block}  of a graph $G$ is a maximal connected subgraph of $G$ without a cutvertex of the subgraph; it may contain 
 cutvertices of $G$. Thus, every block is either a maximal 2-connected subgraph, or a bridge  (with its ends), or an isolated vertex.

 Let $f$ be a proper $t$-coloring of a graph $G$.
For  any two distinct colors $c$ and $d$, we  shall call a path (or cycle) $(c,d)$-{\it colored} or simply $2$-{\it colored} if its edges are alternately colored $c$ and $d$. 
A  colored path (or cycle) $C$ 
 is $c$-{\it alternating} if one of any two consecutive edges in $C$ is colored 
 with color $c$.

 For a vertex $v \in V(G)$,
	we say that a color $i$ {\em appears at $v$
	under $f$} if there is an edge $e$ incident to $v$
	with $f(e) = i$, and we set
	$$f(v) = \{f(e) : e \in E(G) \text{ and $e$ is incident to $v$}\}.$$
Every color in the set $\{1,2,\dots,t\}\setminus f(v)$ is called a {\it missing color} of $f$ at $v$.

Fournier gave a condition for a graph to be Class 1.

\begin{proposition} [Fournier  \cite{Fournier}]
\label{fournier}
Let $G$ be a 
graph. If the subgraph induced by the set of vertices  of degree $\Delta(G)$ in $G$, is acyclic, then $\chi'(G)=\Delta(G)$.
\end{proposition} 
	
A graph $H$ is called a {\it Halin graph} if $H=T\cup C$, where $T$ is a 
plane tree on at least four vertices in which 
no vertex has degree 2, and $C$ is a cycle connecting the leaves of $T$ in the cyclic order determined by 
the  embedding of $T$. Clearly, the subgraph induced by the set of vertices in a Halin graph $G$ of degree at least $4$, is acyclic.
Therefore \cref{fournier} implies the following:

\begin{proposition} 
\label{halin}
Let $G$ be a Halin graph.  
Then  every subgraph $H$ of $G$ with $\Delta(H)\geq 4$ is Class 1.
\end{proposition}

A  graph is said to be {\it chordless} if   in every cycle of $G$  any two nonconsecutive vertices are not adjacent. Chordless graphs were first studied independently by Dirac \cite{dirac} and Plummer \cite{plummer}
in connection with  minimally 2-connected graphs.

A  2-connected graph is called {\it minimally 2-connected}, if for any $e\in E(G)$, the graph 
$G-e$ is not 2-connected. It can be easily verified that a graph is minimally
$2$-connected if and only if it is $2$-connected and chordless.
Note two   properties of chordless graphs  obtained in \cite{dirac}, \cite{plummer} and \cite{machado}.

\begin{proposition} [Dirac \cite{dirac}, Plummer \cite{plummer}]
\label{plummer}
Let $G$ be a   2-connected  chordless graph. Then  any cycle $C$ in $G$ with $|V(C)|\geq 4$ contains  at least two 
non-adjacent vertices whose degrees in $G$ are two.
\end{proposition}

\begin{proposition} [Machado et al \cite{machado}]
\label{machado}
All chordless graphs with maximum degree at least 3 are Class 1. 
\end{proposition}

   A graph is {\it series-parallel} \  if it does not contain a $K_4$-minor. 
Note the following properties of series-parallel graphs (see, for example, \cite{nishizeki}).

\begin{proposition}
\label{prop:series}
If $G$ is a series-parallel graph with maximum degree at least $3$, then $G$ is Class 1 and contains a vertex of degree at most $2$. Moreover, every subgraph $H$ of  $G$ with $\Delta(H)\geq 3$ is Class 1 graph, too.
\end{proposition}

An important type of series-parallel graphs are outerplanar graphs. 
A    graph  is 
{\it outerplanar}, if it can be embedded in the plane in such a way that all its vertices lie 
on the boundary of the outer face. 
It is known that a graph is outerplanar if and only if it contains neither a $K_4$-minor nor a $K_{2,3}$-minor.

The length of a shortest cycle in a graph $G$ is called the {\it girth of $G$}.
Note two properties of planar graphs with large girth.

\begin{proposition} [Bonduelle, Kardo\u s \cite{girth}]
\label{girthplanar1}
Planar graphs of maximum degree 3 and of girth at least 7 have proper 3-colorings.
\end{proposition}

\begin{proposition} 
\label{girthplanar2}
Planar graphs of maximum degree 3 and of girth at least 6 have a vertex of degree at most $2$.
\end{proposition}

\begin{proof}
Without loss of generality we assume that $G$ is connected. Let $F(G)$ denote the set of faces of $G$.
Then $|V(G)|-|E(G)|+|F(G)|=2$ and $\sum_{x\in V(G)}d_G(x)=2|E(G)=\sum_{F\in F(G)}d(F)$, where $d(F)$ is the degree of a face $F$.
If $d(x)\geq 3$ for each vertex $x$ of $G$ and $d(F)\geq 6$ for each face $F$ of $G$, then $2|E(G)|\geq 3|V(G)|$, $2|E(G)|\leq 6|F(G)|$
and so, $2=|V(G)|-|E(G)|+|F(G)|\leq 2|E(G)|/3-|E(G)|+|E(G)|/3=0$, a contradiction
\end{proof}

The definition of the set $\cal A$$(q)$, $q\geq 2$, in the introduction implies  that if $G\in \cal A$$(q)$, 
then  any subgraph of $G$ belongs to $\cal A$$(q)$, too.  

Propositions 2.2 -- 2.7 imply that the set $\cal A$$(3)$ contains all  
chordless graphs, outerplanar graphs, series-parallel graphs   
and  planar graphs of girth at least 7. Furthermore, 
the set $\cal A$$(4)$ contains all Halin graphs and, by the result of Sanders and Zhao \cite{SandersZhao},
the set $\cal A$$(7)$ contains all  planar graphs.

A graph obtained from a cycle by adding a new vertex and joining it to every vertex
in the original graph by an edge is called a {\em wheel graph}.

\begin{proposition}
If $W_n$ is a wheel graph with $n\geq 5$, then $W_n\in \cal A$$(3)$.
\end{proposition}


\begin{proof}
Let $V(W_n)=\{v_0,...,v_n\}$ and $E(W_n)= \{v_1v_2, v_2v_3,...,v_{n-1}v_n, v_nv_1\}\cup$
$\{v_0v_1,...,v_0v_n\}$. 
  Furthermore, let $H$ be an arbitrary subgraph of $W_n$ with $\Delta(H)=3$.
It suffices to consider only the case  when the edges 
$v_1v_2, v_2v_3,...,v_{n-1}v_n, v_nv_1$ belong to $E(H).$

If $H$ does not contain three mutually adjacent vertices of degree 3, then, by Proposition 2.1,  $\chi'(H)=3$.
Suppose now that $H$ contains three mutually adjacent vertices of degree 3, say the vertices 
$v_0, v_1$ and $v_2$.
Let $v_s$ be the third vertex adjacent to $v_0$ in $H$.  The following cases are possible.

{\bf Case 1}. $s=3$,  that is, $v_0$ is adjacent to three consecutive vertices on the cycle $v_1v_2...v_nv_1$.

If $n$ is  odd and $n=2k+1$, then $H$ has a
 proper $3$-coloring $f$ where $f(v_0v_1)=1, f(v_0v_2)=2$, $f(v_0v_3)=3$, $f(v_1v_2)=3$, $f(v_1v_{2k+1})=2$
  $f(v_{2i}v_{2i+1})=1$,  for $i=1,2,...,k$,  and $f(v_{2i+1}v_{2i+2})=2$, for $i=1,2,...,k-1$.
 
 If $n$ is  even  and $n=2k$,  then $H$ has a
 proper $3$-coloring $f$ where $f(v_0v_1)=1, f(v_0v_2)=2$, $f(v_0v_3)=3=f(v_1v_2)$, $f(v_1v_{2k})=2$,
 $f(v_{2k-1}v_{2k})=3$, $f(v_{2i}v_{2i+1})=1$,  for $i=1,...,k-1$, and $f(v_{2i-1}v_{2i})=2$, for $i=2,...,k-1$.

{\bf Case 2}. $3<s<n$. .

If $n$ is  odd  and $n=2k+1$, then $H$ has a
 proper $3$-coloring $f$ where $f(v_0v_1)=1, f(v_0v_2)=2$, $f(v_0v_s)=3$, $f(v_1v_2)=3$, $f(v_1v_{2k+1}=3$,
  $f(v_{2i}v_{2i+1})=1$,  for $i=1,2,...,k$,  and $f(v_{2i+1}v_{2i+2})=2$, for $i=1,2,...,k-1$.
 
If $n$ is  even  and $n=2k$,  then $H$ has a
 proper $3$-coloring $f$ where $f(v_0v_1)=1, f(v_0v_2)=2$, $f(v_0v_s)=3$, $f(v_1v_2)=3$,  $f(v_1v_{2k})=2$,
 $f(v_{2k-1}v_{2k})=3$, $f(v_{2i}v_{2i+1})=1$,  for $i=1,...,k-1$, and $f(v_{2i-1}v_{2i})=2$, for $i=2,...,k-1$.

Thus, any  subgraph $H$ of $G$ with  $\Delta(H)=3$
   admits a proper $3$-coloring. This means that $W_n\in \cal A$$(3)$.  
\end{proof}

For future reference, we also state the following result.

\begin{proposition} [Asratian \cite{asratian09}]
\label{th:bipartite}
   Let $G$ be a  $t$-regular bipartite graph, $t\geq 4$,  and let $f$ and $g$ be two
    proper $t$-colorings of $G$. If $M(f,t)\not=M(g,t)$, then there is a proper  $t$-coloring $\overline f$   of $G$ 
such that    
$\overline f$ is $3$-equivalent to $f$ and 
$|M(\overline f,t)\cap M(g,t)|>|M(f,t)\cap M(g,t)|$. 
\end{proposition}

We shall also need some notation for cycles and paths.
Let $C$ be a cycle in a  graph $G$.
We denote by $\overrightarrow C$ the cycle~$C$ with a given orientation,
and by $\overleftarrow C$ the cycle~$C$ with the reverse orientation.
If $u,v\in V(C)$ then $u\overrightarrow Cv$ denotes the consecutive vertices of~$C$
from $u$ to~$v$ in the direction specified by $\overrightarrow C$.
The same vertices in reverse order are given by $v\overleftarrow Cu$.
We use $u^+$ to denote the successor of~$u$ on $\overrightarrow C$
and $u^-$ to denote its predecessor.
Analogous notation is used with respect to paths instead of cycles.

\section {Main lemmas} 

In this section we consider only Class 1 graphs.

Let $G$ be a  graph 
with $\Delta(G)=t$, 
and let $f$ and $g$ be   different
proper $t$--colorings of  $G$. Denote by  $G(f,g,t)$ (and, respectively,  by $G(g,f,t)$)
 the colored subgraphs induced by the edge subset
$$M(f,t)\bigtriangleup M(g,t)=(M(f,t)\cup M(g,t))\setminus (M(f,t)\cap  M(g,t))$$
where each edge
$e\in M(f,t)\bigtriangleup M(g,t)$ has the color $f(e)$  (has the color $g(e)$, respectively). 

\medskip

In the next lemma 
we shall describe an algorithm for transforming edge colorings
of a graph by a sequence of $(q+1)$-transformations. This algorithm will also form the basis
of the proofs of subsequent lemmas.
First we need the following definition.

\begin{definition}
Let $h_1$ be an improper $t$-coloring of a graph $G$, and  $Q\subset \{1,2,\dots,t\}$ be a subset of colors such that the coloring induced by the edges with colors from $Q$ is improper. 
If $h_2$ is a $t$-coloring of $G$   which agrees with $h_1$ on the
edges colored from $\{1,2,\dots,t\}\setminus Q$, and such that the coloring induced by the edges 
with colors from $Q$ is
proper, then we call $h_2$  a {\em correction of $h_1$ on the set $Q$}.
\end{definition}

\begin{lemma} 
\label{main:lemma1}
Let $G$ be a graph in $\cal A$$(q)$, $q\geq 3$,  with maximum degree $\Delta(G)=t\geq q+2$.
Furthermore, let  $f$ and $g$ be two  different
proper $t$--colorings of  $G$ and assume the subgraph $G(f,g,t)$
 has a component $C$  which is a path of even length. 
Then  there is a proper $t$-coloring $\overline f$ of $G$ such that  
$\overline f$ is $(q+1)$-equivalent to $f$ 
and $|M(\overline f,t)\cap M(g,t)|>|M(f,t)\cap M(g,t)|$.  
\end{lemma}

\begin{proof}
Let $C$ be a component in $G(f,g,t)$ which is a path of even length, 
$C=v_0e_1v_1e_2\dots e_{2m}v_{2m}$ and    
$E(C)=\{e_1, e_{2}, \dots, e_{2m}\}$. Exactly one of the edges $e_1$ and $e_{2m}$ is colored $t$, so
without loss of generality, we assume that $f(e_1)=t$. This implies that  the color $t$ is missing in $g$ at $v_0$,
$f(e_{2m})\not=t$ and
the color $t$ is missing in $f$ at $v_{2m}$.  Therefore,  $d_G(v_0)<t$ and $d_G(v_{2m})<t$.
  Furthermore, $f(e_1)=t$ implies that  $f(e_{2i+1})=t$ for  $i=1,\dots,m-1$.

We will prove the lemma by showing that for some
$k\geq 1$, there is a sequence of proper $t$--colorings $f_0,f_1,\dots,f_k$, where
$f_0=f$, $f_{i+1}$ is obtained from $f_i$ by a $(q+1)$-transformation, for $i=0,1,\dots, k-1$, 
and  
$$M(f_k,t)=(M(f,t)\setminus  E(C))\cup (E(C)\setminus M(f,t));$$ 
that is,  $M(f_k,t)$ and $M(f,t)$  differ only on $C$ and $$|M(f_k,t)\cap M(g,t)|>|M(f,t)\cap M(g,t)|.$$
Thus, $\overline f=f_k$ is the required coloring. 

Suppose  that $f(e_2)=s_0$. Since $d_G(v_0)<t$, there is a color $c_0\not=t$ missing  in $f$ at $v_0$.
We will describe an algorithm to construct the required colorings $f_1, f_2,\dots$.
 The following cases are possible:\bigskip

{\bf Case 1.} {\em The edges of $C$ are colored with two colors, $s_0$ and $t$, under $f$}.

In this case
we can take $k =1$ and  define $f_1$ as follows:
First we exchange the colors $s_0$ and $t$ along $C$ and  leave
the colors of edges in $E(G)\setminus E(C)$  unchanged. 
 Denote the obtained coloring by $f_0'$. 

If  $s_0$ is missing at $v_0$ under $f$ then $f_0'$ is a proper coloring. Take $f_1=f_0'$, and we are done.

If, on the other hand, $s_0$ appears at $v_0$ under $f$, then $c_0\not=s_0$.  
Consider a set $R$  of 
$q-1$ distinct colors  from the set $\{1,2,\dots,t-1\}\setminus \{s_0\}$ which contains  
the color $c_0$. Each vertex of $G$ is incident with at most $q$ edges with colors from the set
$R \cup  \{s_0\}$ under $f'_0$.
Hence, since $G\in \cal A$$(q)$, there is a correction $f_1$ of $f'_0$ on the set $R \cup  \{s_0\}$.
%
Then $f_1$ is a required proper coloring obtained from $f$ by a $(q+1)$-transformation.  
\bigskip

{\bf Case 2.} {\em The edges of C are colored with at least three colors under $f$.}

       We will construct a sequence of proper $t$--colorings
$f_0,f_1,\dots$ and a sequence of auxiliary improper colorings $f'_{0}, f'_{1},\dots$ in the following way.
\bigskip

{\bf Step  0}.  Put $f_0=f$.
Let  $m_0$ be an integer such that
$f(e_{2m_0+2})\not=s_0$
and $P_0=v_0e_1v_1e_2\dots e_{2m_0}v_{2m_0}$ be a  $(t,s_0)$-colored
even path on $C$. Furthermore,   let $s_1=f(e_{2m_0+2})$
and $n_1$ be an integer such that  
$P_1=v_{2m_0}e_{2m_0+1}v_{2m_0+2}\dots e_{2n_1}v_{2n_1}$
is  a  maximal $(t,s_1)$-colored even path on $C$. 
    We denote by $f'_{0}$  
 a   coloring obtained from $f_0$ by
interchanging the colors $s_0$ and $t$ along  the path $P_0$, and
by interchanging the colors $s_1$ and $t$ along the path $P_1$.

Choose a set $R_0\subset \{1,2,\dots,t-1\}$ of $q$  distinct colors  including the colors $s_0, s_1$ and  $c_{0}$.
Since $G\in \cal A$$(q)$ and each vertex  of $G$ is incident with at most $q$ edges with colors from $R_0$,
there is a correction 
$f'_{00}$ of $f'_0$ on the set $R_0$.

If $n_1 = m$, then
put $k=1$,  $f_1=f'_{00}$ and we are done.

Suppose now 
that $n_1 \neq m$.
Then  $f'_{00}$ is an improper coloring
because two edges, $e_{2n_1}$ and  $e_{2n_1+1}$, incident with the vertex $v_{2n_1}$ are colored $t$.
Moreover, since the color $s_1$ is missing in $f'_{0}$ at $v_{2n_1}$,
   there is a color $\beta_0\in R_0$ missing in $f'_{00}$ at $v_{2n_1}$. Now we define a new coloring $f'_{01}$ as follows.
 If $\beta_0=s_1$, then put $f'_{01}=f'_{00}$, otherwise $f'_{01}$ is obtained from $f'_{00}$
by  interchanging the colors  $\beta_0$ and $s_1$ along the unique maximal $(s_1,\beta_0)$-colored
path with origin  $v_{2n_1}$.
\smallskip

\begin{remark} It is possible now that  the edge $e_{2n_{1}+2}$ receives the color $s_{1}$, that is, $f'_{01}(e_{2n_{1}+2})=s_{1}$. 
\end{remark}
\medskip

{\bf Step} $ i+1 (i\geq 0)$. Suppose that  we  have  already defined a set 
$R_i\subset \{1,2,\dots,t-1\}$ of $q$  distinct colors  with $s_i, s_{i+1}\in R_i$, and an integer  $n_{i+1}$,
$0<n_{i+1}<m$, and constructed a proper $t$--coloring
$f_i$ and an improper $t$--coloring $f'_{i1}$ of $G$,  such that
     $f'_{i1}$  satisfies the following conditions:
\medskip

1a) $M(f'_{i1},j)=M(f_i,j)$ for each $j\notin R_i\cup \{t\}$.\medskip

1b) There are no edges incident with $v_{2n_{i+1}}$ of color $s_{i+1}$,
but there are two of color $t$, and at most one of each color $j\neq s_{i+1}$,
$1\leq j\leq t-1$. \medskip

1c) At each vertex other than 
$v_{2n_{i+1}}$ each color appears on at most one edge, and
$ M(f'_{i1},t)=(M(f,t)\setminus \{e_{2j-1}:  j=1,\dots,n_{i+1}\})\cup \{e_{2j}:  j=1,\dots,n_{i+1}\}.$
\medskip

Since every vertex is incident with at most $q$ edges with colors from
$(R_i\cup \{t\})\setminus \{s_{i}\}$ under $f'_{i1}$, there is correction
$f_{i+1}$ of $f'_{i1}$ on the set $(R_i\cup \{t\})\setminus \{s_{i}\}$.
The coloring  $f_{i+1}$ is proper and since it differs from $f'_{i1}$ 
by at most $q$ color classes, 
$f_{i+1}$  differs
from  $f_i$ by at most $q+1$ color classes, that is, 
$f_{i+1}$ is obtained from $f_i$ by a $(q+1)$-transformation.

    In order to construct the next proper coloring $f_{i+2}$,  we   transform $f_{i+1}$ back to $f'_{i1}$ and 
     define  an integer $m_{i+1}$ and a new improper $t$-coloring $f'_{i2}$ as follows: 

\begin{itemize}
\item[(A)] If $f'_{i1}(e_{2n_{i+1}+2})\not=s_{i+1}$ then put $m_{i+1}=n_{i+1}$
and $f'_{i2}=f'_{i1}.$ 

\item [(B)] If $f'_{i1}(e_{2n_{i+1}+2})=s_{i+1}$ (see Remarks 3.1 and 3.2), then let $m_{i+1}$
be  the maximum
integer $j\leq m$ such that 
the path $L_i=v_{2n_{i+1}}e_{2n_{i+1}+1}v_{2n_{i+1}+1}\dots e_{2j}v_{2j}$
is an even $(t,s_{i+1})$-colored
path (with respect to $f'_{i1}$)  on $C$.
 We define   $f'_{i2}$ to be
a coloring  obtained from $f'_{i1}$
by interchanging the colors
$t$ and $s_{i+1}$ along  the path $L_i$. 

{\bf Subcase B1}. {\em $L_i$ includes the edge $e_{2m}$ of $C$.} 

This implies that $f'_{i2}$ is a proper $t$-coloring and it
differs  from $f_{i+1}$ by two color classes.
Take $k=i+2$, $f_{k}=f'_{i2}$ and we are done. 

\medskip

 {\bf Subcase B2}. {\em $L_i$ does not include the   edge $e_{2m}$.} 

This implies that $m_{i+1}<m$,
$L_i=v_{2n_{i+1}}e_{2n_{i+1}+1}v_{2n_{i+1}+1}\dots e_{2m_{i+1}}v_{2m_{i+1}}$,
 and $f'_{i1} (e_{2m_{i+1}+2})\not=s_{i+1}$, and we proceed with the
construction of  $f_{i+2}$ as when (A) holds.
\end{itemize}

In both the cases (A) and (B), the coloring  $f'_{i2}$ is not proper, and has the following properties:
\medskip

2a) There are no edges incident with $v_{2m_{i+1}}$ of color
$s_{i+1}$, but there are two of color $t$, and at most one of each color $j\neq s_{i+1}$,
$1\leq j\leq t-1$. \medskip

2b) At each vertex  $v\not=v_{2m_{i+1}}$ each color appears on at most one edge.\bigskip

Now we use the coloring $f'_{i2}$ and the integer  $m_{i+1}$ to define a new  coloring $f'_{i+1}$
and an integer $n_{i+2}$.

 Let $s_{i+2}$ denote the color of the edge $e_{2m_{i+1}+2}$ under  $f'_{i2}$,  
 and let $n_{i+2}$ be  the maximum
integer $j\leq m$ such that  the path 
$P_{i+2}=v_{2m_{i+1}}e_{2m_{i+1}+1}v_{2m_{i+1}+1}\dots e_{2j}v_{2j}$  is an
even $(t,s_{i+2})$-colored
path (with respect to $f'_{i2}$)  on $C$.
  We denote by $f'_{i+1}$ a $t$--coloring obtained from
$f'_{i2}$ by interchanging  the colors $t$ and $s_{i+2}$ along $P_{i+2}$.\medskip

\noindent {\bf Case 2.1}.  {\em $P_{i+2}$ includes the edge $e_{2m}$.}

 If $f'_{i+1}$ is a proper coloring then take $k=i+2$, $f_{k}=f'_{i+1}$ and we are done. 

If $f'_{i+1}$ is not proper, then two edges of color $s_{i+2}$ are incident to  the vertex $v_{2m_{i+1}}$ 
under the coloring $f'_{i+1}$ (and this is the unique violation under  $f'_{i+1}$). 
To obtain the required proper coloring $f_k$ we define it to be a correction
of $f'_{i+1}$ on 
a set $R$ of $q$ different colors including $s_{i+1}$ and $s_{i+2}$, but not $t$.
This is the required coloring, and we are done.
\bigskip

\noindent {\bf Case 2.2}. {\em $P_{i+2}$ does not include the edge $e_{2m}$.}

Then  $n_{i+2}<m$,
$$P_{i+2}=v_{2m_{i+1}}e_{2m_{i+1}+1}v_{2m_{i+1}+1}\dots e_{2n_{i+2}}v_{2n_{i+2}},$$
  and $f'_{i+1} (e_{2n_{i+2}+2}) \not=s_{i+2}$.
Furthermore, the coloring $f'_{i+1}$ is not proper, and has the following properties:
\medskip

3a) There are no edges incident with $v_{2n_{i+2}}$ of color
$s_{i+2}$, but there are two of color $t$, and at most one of each color $j\neq s_{i+2}$,
$1\leq j\leq t-1$. \medskip

3b) There are no edges  of color
$s_{i+1}$,  at most two edges of color $s_{i+2}$, and at most one of each color $j\notin \{s_{i+1}, s_{i+2}\}$,
$1\leq j\leq t-1$, incident with the vertex $v_{2m_{i+1}}$. \medskip

3c) At each vertex  $v\notin \{v_{2m_{i+1}},v_{2n_{i+2}}\}$ each color appears on at most one edge. \medskip

Now we will construct  two new improper colorings of $G$. 
First we define a set of colors $R_{i+1}$ as follows: If $s_{i+2}\in R_i$, then put $R_{i+1}=R_i$,
otherwise define  $R_{i+1}=(R_i\setminus \{s_i\})\cup \{s_{i+2}\}$. 
Then $R_{i+1}$ contains $q$   colors,
$R_{i+1}\subset \{1,2,\dots,t-1\}$ and  $s_{i+1}, s_{i+2}\in R_{i+1}$.
Since every vertex of $G$ is incident with at most $q$ edges with colors
from $R_{i+1}$ under $f'_{i+1}$, there is a correction $f'_{i+1,0}$
of $f'_{i+1}$ on the set $R_{i+1}$.

The condition (3a) implies that  at most $q-1$ edges with colors from
$R_{i+1}$ are incident with  $v_{2n_{i+2}}$.
So there is a color $\beta_{i+1}\in R_{i+1}$  
missing  at $v_{2n_{i+2}}$ under $f'_{i+1,0}$. 
We define a new coloring $f'_{i+1,1}$ as follows.  
 If $\beta_{i+1}=s_{i+2}$, then put $f'_{i+1,1}=f'_{i+1,0}$, otherwise $f'_{i+1,1}$ is obtained from $f'_{i+1,0}$
by  interchanging   the colors $\beta_{i+1}$ and $s_{i+2}$ along 
 the unique maximal $(s_{i+2},\beta_{i+1})$-colored path in $G$ with origin  $v_{2n_{i+2}}$. 
\smallskip

\begin{remark}
It is possible now that $f'_{i+1,1}(e_{2n_{i+2}+2})=s_{i+2}$.
\end{remark}

Thus, at Step $i+1$ of the algorithm we constructed the colorings $f_{i+1}, f'_{i+1,0}$, $f'_{i+1,1}$ and defined the integer $n_{i+2}$.
Go to Step $i+2$.

\medskip

Since $G$ is a finite graph, it follows that we can repeat the steps in the above algorithm
to obtain, for some $k\geq 1$, the required coloring $\overline f=f_k$.
 The proof of the lemma is complete.
   \end{proof}

 \begin{lemma}
\label{main:lemma2}
Let $G$ be a  graph in $\cal A$$(q)$, $q\geq 3$, with maximum degree $t\geq q+2$.
Furthermore, let $f$ and $g$ be two  different
proper $t$--colorings of  $G$ such that  the subgraph $G(f,g,t)$
 has a component  which is a path of odd length. 
Then there are two proper $t$-colorings $\overline f$ and $\overline g$ of $G$ such that    
$\overline f$ is $(q+1)$-equivalent to $f$, $\overline g$ is $(q+1)$-equivalent to $g$ and 
$|M(\overline f,t)\cap M(\overline g,t)|>|M(f,t)\cap M(g,t)|$. 
\end{lemma}

\begin{proof} 
Let $H$ be a component in $G(f,g,t)$ which is a path of odd length.
Then the definition of $G(f,g,t)$ implies that the degrees of the origin and 
terminus of $H$ are less than $t$.
We may assume that the first edge of $H$ is not colored $t$ under $f$
since otherwise  we may switch the roles of the colorings $f$ and $g$ 
and obtain an analogous result by considering 
an odd path $H$ in the subgraph $G(g,f,t)$ instead of $G(f,g,t)$.

Let $H=u_0e_0v_0e_1v_1e_2\dots v_{2m-1}e_{2m}v_{2m}$ and $f(e_0)=c_0\not=t$,
so $f(e_{2i-1})=t$, for $i=1,2,\dots,m$. Put $\overline g=g$.
If $H$ consists of one edge $e_0=u_0v_0$,  then, by the definition of $G(f,g,t)$, the color $t$ is missing at $u_0$ and $v_0$ 
under $f$. Therefore, we define $\overline f$ from $f$
by  recoloring  the edge $u_0v_0$ with color $t$.   

Suppose now that $H$ consists of at least three edges.  Let $C$ denote the path $v_0v_1v_2 \dots v_{2m}$ and assume $f(e_2)=s_0$. 
The following cases are possible.
\medskip

{\bf Case 1}. {\em $C$ is colored with two colors, $s_0$ and $t$.}

In this case we may construct the required  $t$-coloring $\overline f$ from $f$ by
 exchanging the colors of the edges $e_{2i}$ and $e_{2i+1}$, for $i=0,1,\dots, m-1$,
and then color the edge $e_{2m}$ with color $t$. 
\medskip

{\bf Case 2}. {\em $C$ is colored with at least 3 colors.}

With a minor modification, the algorithm described in the proof of Lemma \ref{main:lemma1}
 can be used for constructing the required  coloring $\overline f$.
 In fact, the modification concerns only  the improper coloring $f'_0$ at Step 0. More precisely, we will 
instead
define  $f'_0$ 
from $f_0=f$ by first recoloring the edge $e_0=u_0v_0$ with color $t$, then
interchanging the colors $s_0$ and $t$ along  the path $P_0=v_0e_1v_1e_2\dots e_{2m_0}v_{2m_0}$, and finally
interchanging the colors $s_1$ and $t$ along the path $P_1=v_{2m_0}e_{1+2m_0}v_{1+2m_0}\dots e_{2n_1}v_{2n_1}$.

Clearly,  the color $c_0=f(e_0)$ is missing at $v_0$ under $f'_0$. Therefore, if we continue to perform, without any changes,
 the algorithm described in the proof of Lemma \ref{main:lemma1} to the path $C=v_0v_1v_2\dots v_{2m}$,
we will eventually construct 
the required coloring $\overline f$ by applying $(q+1)$-transformations only. 
\end{proof}

\begin{lemma} 
\label{main:lemma3} 
   Let $G$ be a  graph in $\cal A$$(q)$, $q\geq 3$, with maximum degree $t\geq q+2$ and let $f$ and $g$ be two
    proper $t$-colorings of $G$. 
   Suppose that the subgraph $G(f,g,t)$ contains a $t$-alternating cycle $C$ 	that is colored with
	at most $q+1$ colors, or the degree of one of the vertices of $C$ is less than $t$ in $G$. 
  Then 
there is a proper  $t$-coloring $\overline f$   of $G$ 
such that 
$\overline f$ is $(q+1)$-equivalent to $f$ and 
$|M(\overline f,t)\cap M(g,t)|>|M(f,t)\cap M(g,t)|$. 
\end{lemma}

\begin{proof}
Let
$C=v_0e_1v_1e_2\dots v_{2m-1}e_{2m}v_{2m}$, where $v_{2m}=v_0$, and let $E(C)=\{e_1,\dots, e_{2m}\}$. 
 Without loss of generality, we can assume that  
$f(e_{2i-1})=t$ for  $i=1,\dots, m$ and that 
$f(e_2)=s_0$. The following  cases are possible: 

\medskip

 {\bf Case 1.} {\em $C$ is colored with  at most $q+1$ different colors  under $f$.}

 Assume that the edges of $C$ are colored with $p+1$  colors, $t, s_0,s_1, \dots,s_{p-1}$, where $p\leq q$.
We construct $\overline f$
  as follows: 
 first exchange the colors of the edges $e_{2i-1}$ and $e_{2i}$ for $i=1,\dots, m$,
 and retain the color of every edge in $E(G)\setminus E(C)$. The obtained coloring we denote by $f_0'$.
 Choose $q$ distinct colors $b_1,b_2,\dots, b_{q}$ from the set $\{1,2, \dots,t-1\}$ 
including the colors $s_0,s_1,\dots,s_{p-1}$.
 Let $D$ be  the set of edges in $E(G)$ colored with colors $b_1,\dots,b_q$ under $f_0'$. 
Clearly, every vertex of $G$ is incident with at most $q$ edges from $D$, so since $G\in$$ \cal A$$(q)$,
we can  properly color the edges in $D$ with colors $b_1,\dots,b_q$ to obtain the required coloring $\overline f$.

\medskip

{\bf Case 2.} {\em 
A vertex in $C$ has degree less than $t$ in $G$.}

Without loss of generality, we can assume that  $d_G(v_0)<t$ and a color, say $c_0$, is missing at $v_0$ under $f$.
 Then using (without any changes)   the algorithm described in
the proof of Lemma 3.1  we can transform the coloring $f$ along the cycle $C$ and construct  for some $k\geq 1$, proper colorings  $f_0,f_1,\dots,f_k$ such that $f_0=f$, 
  $$M(f_k,t)=(M(f,t)\setminus  (E(C))\cup (E(C)\setminus M(f,t))$$
 and $f_{i+1}$ is obtained from $f_i$ by a $(q+1)$-transformation, for  $i=0,1,\dots, k-1$. 
  Note that the coloring $f_k$
 will be obtained at the step where the edge $e_{2m}$ receives the color $t$.
 Then $\overline f=f_k$ is the required coloring
  since $|M(f_k,t)\cap M(g,t)|>|M(f,t)\cap M(g,t)|$.
\end{proof}

\begin{lemma} 
\label{main:lemma4}
Let $G$ be a graph in $\cal A$$(q)$, $q\geq 3$, with maximum degree $t\geq q+2$,
and let
 $f$ and $g$ be two different proper $t$-colorings of $G$. Assume that  the subgraph $G(f,g,t)$
 contains a component $C$ which is a cycle where the edges are colored with at least $q+2$ colors, all vertices have degree $t$ 
 in $G$,  
 and there is a $t$-alternating path $P$ in $G$ satisfying the following conditions: 

\begin{itemize}
	
	\item $C$ and $P$ have only one common vertex which is the terminus of $P$,

	\item the origin of $P$ has degree less than $t$, 
	
	\item if $P$ has an odd length, then the color $t$ is missing at the origin of $P$. 

\end{itemize}

  Then 
there is a   proper  $t$-coloring $\overline f$   of $G$ 
such that    
$\overline f$ is $(q+1)$-equivalent to $f$, $M(\overline f,t)$ and $M(f,t)$ differ only on $C$, and 
$|M(\overline f,t)\cap M(g,t)|>|M(f,t)\cap M(g,t)|$. 
\end{lemma}

\begin{proof}
	Let  $P=u_0u_1 \dots u_l$, $l\geq 1$, and $C=v_0v_1 \dots v_{2m}$ where $v_{2m}=v_0=u_l$ and
	$v_0$ is the only common vertex of $C$ and $P$. 
Throughout we assume that $C$ and $P$ are oriented, 
	so that the edge $e_1=v_0v_1$ succeeding $v_0$ on $C$ is colored $t$. 
	
	The idea of the proof
	is to apply the algorithm from the proof of Lemma \ref{main:lemma1}, 
	referred to as algorithm $A$ henceforth, first along $P$, then 
	continue along     $C$,
	and then finally continue the process along $\overleftarrow P$, that is, the path $P$ where edges 
	are traversed in opposite order. 
	To describe this process in a unified way, we make some modifications in  algorithm $A$, and  apply 
	the obtained new algorithm, denoted algorithm $A'$, directly to the whole trail $P \cup C \cup \overleftarrow P$. 
	We consider two main cases:
	\medskip

	(i) {\it The first edge of $P$ is colored $t$.}

	Consider the trail 
	$$W=w_0e'_1w_1e'_2w_3 \dots e'_{2m+2l}w_{2m+2l}$$
	    with vertices $w_0, w_1,\dots,w_{2m+2l}$ and edges
	 $e'_1, e'_2, e'_3,\dots,e'_{2m+2l}$ where
 $w_i=u_i$ for $i=0,1,\dots,l$, $w_{l+i}=v_i$, for $i=0,1,\dots,2m$,  and 
	$w_{2m+2l-i}=u_i$, for $i=0,1,\dots,l$.  
	
	We denote by $s_0$ the color of the second edge of $P$.
	Let us describe some  differences  and  similarities between the  algorithms $A$ and $A'$  in some detail.

\begin{itemize}

     \item The first difference is that algorithm $A'$ processes $W$ rather than $C$, and thus
     the vertices from the set $V(P)\cup V(C)$ and the edges from the set $E(P)\cup E(C)$  in algorithm $A$
     will be replaced in $A'$ with vertices $w_0, w_1, \dots,w_{2m+2l}$ and edges $e'_1, e'_2,\dots, 
e'_{2m+2l}$, respectively.     
	
%
	On the other hand, for all considered subpaths of $W$ that we consider in algorithm $A'$, 
	we use the notation $P_i$ and $L_i$, just as in algorithm $A$.

	\item The start of $A'$ is similar to the start $A$: we construct a coloring  $f'_{0}$   by
interchanging the colors $s_0$ and $t$ along  the path $P_0$ on $W$ with origin $w_0$, and
by interchanging the colors $s_1$ and $t$ along the path $P_1$ on $W$ with origin $w_{2m_0}$.
($P_0$ and $P_1$ are paths on $W$ because  the edges of $C$ as well as of $W$ are colored by at least $q+2$ colors ($q\geq 3$)
in the initial coloring $f$.)

\item The second difference of $A'$ compared to the algorithm $A$  is when
	we enter the cycle $C$ of $W$.  
	Either $w_l$ is the terminus of a
	maximal 2-colored path $P_{i+1}$ (or $L_i$) considered by the algorithm $A'$,
	or $w_l$ is an internal vertex of such a maximal 2-colored 
	path $P_{i+1}$ (or $L_i$) that is considered by the algorithm. 
	In both cases $l$ is an even integer and the color of the edge $w_{l-1}w_l$ is distinct from $t$.

	{\bf Case 1.} {\em $w_l$ is the terminus of a 
	maximal 2-colored path $P_{i+1}$ considered by  algorithm $A$.}

	After considering $P_{i+1}$ as in  algorithm $A$,
  the algorithm $A'$ then continues with the next trail   
  $P_{i+2}$ (or $L_i$) which starts at $w_l$ and is contained in $C=w_lw_{l+1}\dots w_{l+2m}$. 
	More precisely, $P_{i+2}$ is either  a path in the cycle $C$ or the whole cycle 
itself.\footnote{It is possible that after transformations  involving the subpaths $P_0, P_1,\dots,P_{i+1}$  of the path $w_0w_1\dots,w_l$
	the cycle $C$  will be  colored with two colors.}
In the first case, we proceed exactly as described in the algorithm $A$.
	In the second case we interchange colors on the whole cycle $C$, and then take the next path $P_{i+3}$ (or $L_{i+1}$)
	to start at $w_{l+2m}=w_l= v_0$ and be contained in $\overleftarrow P$.

	{\bf Case 2.} {\em $w_l$ is an internal vertex of a maximal 2-colored path $P_{i+1}$
	that is considered by the algorithm.}

	In this case  the considered path $P_{i+1}$ ends at some vertex of $C$ which  is distinct  from $w_{l+2m}$. 
	This is evident if $i=0$ and 
	this follows for $i\geq 1$ from the fact that 
	the adjacent edges $w_{l-1}w_l$ and 
	$w_{l+2m-1}w_{l+2m}$  have distinct colors under every proper 
	$t$-coloring $f_j$ and every improper $t$-coloring $f'_j$ constructed before considering the path $P_{i+1}$.
	 
	 Since the terminus of $P_{i+1}$   is distinct 
	from $w_{l+2m}$, we proceed exactly as described in  algorithm $A$.  

	\item The third difference compared to the algorithm $A$  is when we exit the cycle $C$ of $W$.

	As in the preceding point, the vertex $w_{l+2m}=w_l$
	 might be the terminus of a $2$--colored cycle or path $P_{i+1}$ 
	(or $L_i$) that the algorithm considers.	
	In this case, we process $P_{i+1}$ as in  algorithm $A$ 
	and then continue with the next path $P_{i+2}$ (or $L_i$) along $W$ that has origin $w_{l+2m}$ and 
	is contained in $\overleftarrow P$.
	On the other hand, it might be the case that $w_{l+2m}$ is an internal vertex of some path 
	$P_{i+1}$ when leaving the cycle $C$, in which
	case we proceed exactly as described in  algorithm $A$. 
	
	\item At the final step the algorithm $A'$ is similar with algorithm $A$.

	\end{itemize}

	Apart from these differences, since we traverse vertices and edges exactly as in the algorithm
	$A$ in the proof of Lemma \ref{main:lemma1},
	except that we follow the trail $W$ rather than a path (and edges of $P$ are traversed twice),
	the same arguments as in that proof yield that we obtain the coloring $f_k$ which has the required properties.
	In particular, the sequence $f_0, f_1, \dots, f_k$ of proper $t$-colorings that the algorithm $A'$ produces satisfy
	that $f_{i+1}$ is obtained from  $f_{i}$ by a $(q+1)$-transformation.
	Since the edges of $P$ are traversed twice, the edges of $P$ that are colored $t$ under
	 the initial coloring $f$,	will also be colored $t$ in the final coloring $f_k$ after completing the algorithm $A'$. 
	Thus,   $M(f_k,t)=(M(f,t)\setminus  E(C))\cup (E(C)\setminus M(f,t))$, which implies that $M(f_k,t)$ and $M(f,t)$  
differ only on $C$ and $|M(f_k,t)\cap M(g,t)|>|M(f,t)\cap M(g,t)|$. By setting $\overline f=f_k$ 
we obtain the required coloring.

\bigskip

(ii) {\it The first edge of $P$ is colored with a color $s_0\not=t$.}

	If this holds, then we form a new graph $G'$ from $G$ by adding a new vertex $a_0$ to $G$ and joining
	it to $w_0$. Then $G' \in {\cal A}(q)$, and we color $a_0w_0$ by the color $t$, and treat this coloring
	of $G'$ as the initial coloring.
	Instead of considering the trail $W$ we apply the algorithm $A'$ from part (i)
	to the trail
	$$W' =a_0e'_0w_0e'_1w_1e'_2w_3 \dots e'_{2m+2l+1}w_{2m+2l+1}$$ in $G'$,
	where $w_{2m+2l+1}=a_0$ and $w_0,w_1, \dots, w_{2m+2l}$ are the same as in part (i).

	Thus we proceed exactly as in part (i) to obtain a required proper coloring $f_k$ of $G'$. 
	Then we simply
	remove the vertex $a_0$ and the edge $a_0w_0$ and take $\overline f$ as the restriction of $f_k$ of $G$. 
	This completes the proof of the lemma.
\end{proof}

\begin{corollary} 
\label{main:corollary} 
Let $G$ be a graph in $\cal A$$(q)$, $q\geq 3$, with maximum degree  $t\geq q+2$.
If $f$ and $g$ are two  different
proper $t$--colorings of  $G$ and  the subgraph $G(f,g,t)$
 has a component $C$  which is a cycle containing  a cutvertex of $G$,
 then  there is a proper $t$-coloring $\overline f$ of $G$ such that $\overline f$ is $(q+1)$-equivalent to $f$ 
and $|M(\overline f,t)\cap M(g,t)|>|M(f,t)\cap M(g,t)|.$
\end{corollary}

\begin{proof}
Let $C=v_0e_1v_1e_2\dots v_{2m-1}e_{2m}v_{2m}$, where $v_{2m}=v_0$, 
and  assume that $v_0$ is a cutvertex. 
 We can also assume that the edges of $C$ are colored with at least $q+2$ colors and all vertices on $C$ have the maximum degree in $G$, because otherwise the result follows from Lemma \ref{main:lemma3}.

Let $B$ denote a 2-connected block of $G$ containing $C$. 
Since $v_0$ is a cutvertex, 
$d_B(v_0)<t$, and so
there is a color $c_0\not=t$ that is missing in $B$ at $v_0$. 
Since $d_G(v_0)=t$, there is another block, say $B_1$,
containing an edge $v_0u_1$ of color $c_0$. Consider in $B_1$ a (unique) maximal $(c_0,t)$-colored path with origin $v_0$
and with first edge $v_0u_1$, denoted  $P$. Then the cycle $C$ and the path $P$ satisfies the conditions of Lemma 
\ref{main:lemma4}, and so the result follows from Lemma \ref{main:lemma4}. 
\end{proof}

Next, we consider a special type of ear decomposition that will prove useful.
An {\em ear} of a graph $H$ is a 
path whose endpoints are in $H$, and whose internal vertices have degree 2.

Let $B$ be a graph properly edge-colored with colors $1,2,\dots,t$, where $t=\Delta(B)$, and $C$ be a $t$-alternating  cycle in $B$.
We say that a subgraph $H$ of $B$ is a {\em $t$-alternating $(\text{ear},C)$-subgraph} 
if $H= P_0\cup \dots \cup P_{n}$
for some $n\geq 0$,
where $P_0=C$, and 
for $n \geq 1$, $P_i$ is a $(c_i,t)$-colored
path which is an ear of the graph  $ P_0\cup P_1\cup \dots \cup P_{i}$, for $i=1,\dots,n$,
where $c_i\leq t-1$.
We also say that $P_0,\dots,P_n$ is a  {\em $t$-alternating ear decomposition} of $H$
 with respect to the cycle $C$.

For a vertex $v$  in a $t$-alternating $(\text{ear},C)$-subgraph $H$, {\em a  $(C,v)$-path} 
 is a path $Q$ in $H$ from a vertex of $C$ to the vertex $v$ such that $|V(Q)\cap V(C)|=1$.

\begin{lemma}
\label{lem:ear1}
	Let $B$ be a non-regular $2$-connected properly $t$-colored graph, where $t=\Delta(B)$. 
	If $C$ is a cycle of $G$ where all vertices
	have degree $t$, then there is
	a $t$-alternating $(\text{ear},C)$-subgraph $H=P_0\cup...\cup P_n$  of $B$ with $P_0=C$ and $n\geq 0$,
	such that either
\begin{itemize}

	\item[(i)] $H$ contains a vertex $x$ of degree less than $t$ in $B$, or

	\item[(ii)] there is a $t$-alternating path $P_{n+1}$ from a vertex $y$ of $H$ to a vertex $x$ of degree
	less than $t$, such that $y$ is the only common vertex of $P_{n+1}$ and $H$, and $P_{n+1}$ contains the edge colored $t$ incident with $x$,
	if such an edge exists.
\end{itemize}
\end{lemma}
\begin{proof}
Let $H = P_0 \cup P_1 \cup \dots \cup P_{n}$  be an edge-maximal $t$-alternating 
	$(\text{ear},C)$-subgraph of $B$ where $P_0=C$ and $n\geq 0$. We will show that if the condition (i) does not hold, 
	  then the condition (ii)  must hold. Assume that all vertices in $H$
	have degree  $t$ in $B$. The  subgraph of $B$ induced by $V(H)$ is not $t$-regular, 
	because $B$ is non-regular and connected.
	Thus there is some edge $e \in E(B) \setminus E(H)$ that is incident with a vertex of $H$. 
	Suppose that $e$ is colored $c_1$, and let $P_{n+1}$ be a
	maximal $(c_1,t)$-colored path containing $e$, whose internal vertices are not contained in $H$. The maximality of $H$
	implies that 
	$P_{n+1}$ ends at some vertex $x$ with $d_B(x)<t$ not contained in $H$.
	Clearly, $P_{n+1}$ contains the edge colored $t$ incident with $x$,
	if such an edge exists.
\end{proof}

Next, we have the following.

\begin{lemma}
\label{lem:ear2}
	Let  $H_n$ be a graph and $C, P_1,\dots, P_n$ be  a $t$-alternating ear decomposition  of $H_n$
	with respect to the cycle $C$, $n\geq 1$.
	Then for every vertex $v$ of $P_1\cup \dots \cup P_n$, 
	there is a $t$-alternating  $(C,v)$-path  $F$ in $H_n$ such that the
	edge colored $t$ that is incident with $v$ is contained in $F$.
\end{lemma}
\begin{proof}
	The proof is by induction on the number of ears $n$.  The result is trivial if $n=1$, that is, $H_1$ has only one ear.

Assume that the proposition of the lemma is true for all graphs permitting a $t$-alternating ear decomposition with $k$ ears,
 for some $k\geq 1$.
 Consider a graph $H_{k+1}$ that has a  $t$-alternating ear decomposition $C, P_1,\dots, P_{k+1}$ with respect to the cycle $C$.
Let $v$ be a vertex in  $P_1\cup \dots \cup P_{k+1}$ and let $H_k=C\cup P_1\cup \dots \cup P_k$. 
Then  $C, P_1, \dots ,P_k$ is a $t$-alternating ear decomposition  of $H_k$ with respect to $C$.
 
 If $v$  belongs to $P_1\cup \dots \cup P_k$, then, by the induction hypothesis, there is a $t$-alternating  $(C,v)$-path  $F$ in $H_k$
 (and, therefore, in $H_{k+1}$) such that the edge colored $t$ that is incident with $v$ is contained in $F$.
 
If $v$  does not belong to $P_1\cup \dots \cup P_k$, then $v$ is an internal vertex of $P_{k+1}$. 
Let  $z_1, z_2$ be the origin and terminus of $P_{k+1}$, 
respectively. 
Denote by $D_i$  the subpath of $P_{k+1}$ from $z_i$ to $v$, $i=1,2$. One of the edges on $P_{k+1}$ incident to $v$ is colored $t$;
without loss of generality we assume that this edge belongs to $D_1$.

If $z_1\in V(C)$ then $F=D_1$ is the required $(C,v)$-path. 
If, on the other hand, $z_1\notin V(C)$, then let $z_1y_1$ be the edge colored $t$
that is incident to $z_1$. By the induction hypothesis, there is a $t$-alternating $(C, z_1)$-path 
$Q$ in $P_1\cup \dots \cup P_k$
containing the edge $z_1y_1$. Thus $Q D_1$ is the required $t$-alternating $(C,v)$-path.
\end{proof}

\section{Main results and proofs}

Using the lemmas from the preceding section, we shall prove \cref{main:th}.

\begin{proof}[Proof of Theorem \ref{main:th}]
Let $\cal B$$(q)$  denote the set of all Class 1 graphs
with maximum degree at least $5$ where every block is either a bipartite graph,
or a $(q+1)$-degenerate graph from the set
$\cal A$$(q)$. 
The proof is  by induction on the maximum degree $\Delta(G)$.  The result is evident if $\Delta(G) \leq q+1$.

Now suppose that 
$G$ is a graph  in $\cal B$$(q)$ with maximum degree $t\geq q+2$ and that the induction
hypothesis holds for all    graphs in $\cal B$$(q)$ with maximum degree $t-1$.
Clearly, $G\in \cal A$$(q)$, since a bipartite graph 
belongs to $\cal A$$(q)$, for every $q\geq 3$.

Let $f$ and $g$ be two distinct proper $t$--colorings of $G$.

\medskip

{\bf Case 1}. $M(f,t)=M(g,t)$:

In this case  the graph $G'=G-M(f,t)$ is a graph  in $\cal B$$(q)$ with maximum degree $t-1$. Let
$f'$ and $g'$ be the two distinct proper $(t-1)$--colorings of
$G'$ induced by $f$ and $g$, respectively.  Then, since by the
induction hypothesis $\chi'_{trans}(G')\leq q+1$, the inequality $\chi'_{trans}(G)\leq q+1$ must be true, too.
\smallskip

{\bf Case 2}:
$M(f,t)\neq M(g,t)$. 

 Since $G(f,g,t)$ has maximum degree at most two,
every component  of the subgraph $G(f,g,t)$ is either a cycle, or a path. We consider some different cases.

If a component $C$ in $G(f,g,t)$ is a path, then we apply Lemma \ref{main:lemma1} or \ref{main:lemma2}
  to obtain two  proper  $t$-colorings $\overline f$  and $\overline g$ 
such that    
$\overline f$ is $(q+1)$-equivalent to $f$, $\overline g$ is $(q+1)$-equivalent to $g$ and 
$|M(\overline f,t)\cap M(\overline g,t)|>|M(f,t)\cap M(g,t)|$. 

If, on the other hand, $G(f,g,t)$ contains a cycle colored with at most $q+1$ colors, 
or containing 
 a vertex with degree less than $t$, then, by Lemma \ref{main:lemma3}  
there is a  proper $t$-colorings  $\overline f$ of $G$ such that  $\overline f$ is $(q+1)$-equivalent to $f$ and 
$|M(\overline f,t)\cap M(g,t)|>|M(f,t)\cap M(g,t)|$.

Suppose now instead that a  component $C$
of the subgraph $G(f,g,t)$ is a cycle  where the edges are colored with at 
least $q+2$ colors and all vertices have the degree $t$ in $G$.
If $C$ contains a cutvertex of $G$, then, by Corollary \ref{main:corollary}, 
there is a  proper  $t$-coloring $\overline f$  of $G$
as in the preceding paragraph. 
If, on the other hand, $C$ does not  contain a cutvertex of $G$, then 
let $B$ be the 2-connected block of $G$ containing $C$. We consider two subcases.
\smallskip

{\bf Subcase 2.1}. {\em $B$ is a $(q+1)$-degenerate graph from the set $\cal A$$(q)$.}

The conditions imply that $B$ is not $t$-regular and contains a vertex of degree less than $t$ in $V(B)\setminus V(C)$.
We will show that there is 
 a $t$-alternating path $Q$ in $B$ 
 satisfying the following three conditions:  the origin of $Q$ is the unique common vertex with $C$,  
 the degree of the terminus of $Q$ is less than $t$, and 
 if $Q$ has an odd length, then the color $t$ is missing at  the terminus of $Q$.

It follows from Lemma \ref{lem:ear1} that there is a $t$-alternating (ear, $C$)-subgraph $H = C\cup P_1 \cup \dots \cup P_n$ in $B$
such that either
\begin{itemize}

	\item[(i)] $H$ contains a vertex $x$ of degree less than $t$, or

	\item[(ii)] there is a $t$-alternating path $P_{n+1}$ from a vertex $y$ of $P_n$ to a vertex $x$ of degree
	less than $t$, such that $y$ is the only common vertex of $P_{n+1}$ and $H$, and $P_{n+1}$ contains the edge colored $t$ incident with $x$,
	if such an edge exists.

\end{itemize}

If (i) holds, then it follows directly from Lemma \ref{lem:ear2} that there is a
$t$-alternating $(C,x)$-path $Q$ in $H$ and, therefore, in $B$. 
If (ii) holds,
then we can apply Lemma \ref{lem:ear2} to the vertex $y$ to obtain  a $(C,y)$-path $Q'$ in $H$, 
and  then add $P_{n+1}$ to $Q'$ to obtain the required $(C,x)$-path $Q$.

Thus, in both cases we obtain a $t$-alternating path $P =  \overleftarrow Q$ satisfying the conditions of Lemma \ref{main:lemma4}:
 $C$ and $P$ have only one common vertex which is the terminus of $P$,
 the origin of $P$ has degree less than $t$, and
	if $P$ has  odd length, then the color $t$ is missing at the origin of $P$. 

Hence, by Lemma \ref{main:lemma4}, 
there is a   proper  $t$-coloring $\overline f$   of $B$ 
such that    
$\overline f$ is $(q+1)$-equivalent to $f$,
$|M(\overline f,t)\cap M(g,t)|>|M(f,t)\cap M(g,t)|$, and $M(\overline f,t)$ and $M(f,t)$ differ only on $C$.

\smallskip

{\bf Subcase 2.2}. {\em $B$ is a bipartite graph.}

If $B$ has a vertex of degree less than $t$, then it is not regular. In the same way
as in the Subcase 2.1, one can show that there is a proper  $t$-coloring $\overline f$   of $B$ 
such that    
$\overline f$ is $(q+1)$-equivalent to $f$,
$|M(\overline f,t)\cap M(g,t)|>|M(f,t)\cap M(g,t)|$, and $M(\overline f,t)$ and $M(f,t)$ differ only on $C$. 

If $B$ is a $t$-regular graph, then 
the existence of a  proper  $t$-coloring $\overline f$   as in the preceding paragraph
follows from \cref{th:bipartite}.
\bigskip

We have thus proved that if $M(f,t)\not=M(g,t)$, then in all cases we can construct 
two  proper  $t$-colorings $\overline f$  and $\overline g$ of $G$
such that    
$\overline f$ is $4$-equivalent to $f$, $\overline g$ is $4$-equivalent to $g$ and 
$|M(\overline f,t)\cap M(\overline g,t)| > |M(f,t)\cap M(g,t)|$. 
%
This implies that there exist two  proper $t$-colorings $f^*$ and $g^*$ of $G$ such that 
$f^*$ is 
$4$-equivalent to $f$, $g^*$ is $4$-equivalent to $g$ and  $M(f^*,t)=M(g^*,t)$.
Then, as in the Case 1, $\chi'_{trans}(G)\leq 4$. 
\end{proof}

Let us state some conseqences of Theorem \ref{main:th}.

\begin{corollary}
\label{th:B}
Let $G$ be a  graph with  $\Delta(G)\geq 4$, where every block is either a bipartite graph,
a chordless graph,  a series-parallel graph,  a wheel graph, 
or a planar graph of girth at least 7.   
 Then $\chi'_{trans}(G)\leq 4$.
\end{corollary}

\begin{proof}
The bound is evident if $\Delta(G)=4$. If $\Delta(G)\geq 5$,
  Propositions 2.2--2.8 imply that  
  all  chordless graphs,  series-parallel graphs,   wheel graphs and  planar graphs of girth at least 7
 are $4$-degenerate graphs in the set $\cal A$$(3)$.  
So the corollary follows from  \cref{main:th} under $q=3$.
 \end{proof}

\begin{figure}
\centering
\includegraphics[width=10cm]{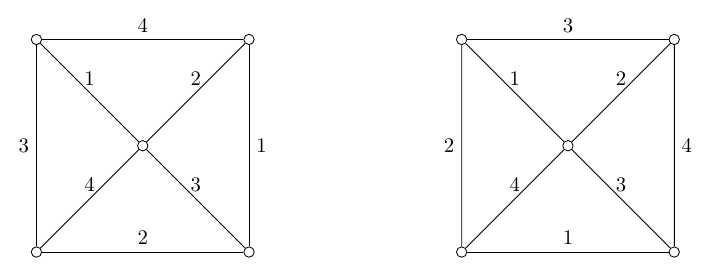}
\caption{A Halin graph $H$ with $\chi'_{trans}(H)=4$.}
\end{figure}

\begin{corollary}

(i) If $G$ is an arbitrary planar graph, then $\chi'_{trans}(G)\leq 8$.
\smallskip

(ii) If $G$ is a Halin graph,  then $\chi'_{trans}(G)\leq 5$,
\smallskip

(iii) If $G$   is either a series-parallel graph, a  chordless graph, 
or a planar graph of girth at least 7,
then $\chi'_{trans}(G)\leq 4$.
\end{corollary}

\begin{proof} 
 The bound  in (i) is evident for planar graphs with maximum degree at most $8$.
Let $G$ be a planar graphs with maximum degree at least $9$.
By the result of Sanders and Zhao [25], $G\in \cal A$$(7)$. The graph $G$ is $5$-degenerate,
since every subgraph of $G$ contains a vertex  of degree at most 5. Then, by Theorem \ref{main:th} 
(under $q=7$), $\chi'_{trans}(G)\leq 8$.

The bound in (ii) follows from Theorem \ref{main:th} under $q=4$, because $G$ is a $3$-degenerate graph in $\cal A$$(4)$.

The bound in (iii) follows from Corollary  4.1.
\end{proof}

The next example shows that there are Halin graphs $H$ with $\chi'_{trans}(H)=4$ and $H\notin \cal A$$(3)$.

\begin{example}
\label{example2}
Consider the  Halin graph $H$ in Figure 1 with
	two different proper $4$-edge colorings $f$ 
	(to the left) and $g$ 
	(to the right).
	It is not difficult to verify that a
	subgraph $H(t_1,t_2,t_3)$ induced by the  set of
	edges $M(f,t_1)\cup M(f,t_2)\cup M(f,t_3)$ 
	gives the same partition of  edges of
	$H(t_1,t_2,t_3)$ into  matchings, for any $1\le t_1 < t_2 < 
	t_3\le 4$.
	Hence $g$ cannot be  obtained from $f$ even by a sequence of
	transformations each of which uses  at most three color classes, that is, $\chi'_{trans}(H)=4$.
	Note further that $H\notin \cal A$$(3)$  since $H-e$ has no proper 3-coloring for any edge $e$ incident to the unique
	 vertex of degree 4 in $H$.
\end{example}

The graph $H$ in Figure 1 shows that the bound for wheel graphs
in Corollary \ref{th:B}  is sharp.
Our next  example shows  that the bound for  series-parallel, and even outerplanar, graphs 
in Corollary \ref{th:B}  is sharp, too.

 \begin{figure}
\centering
\includegraphics[width=10cm]{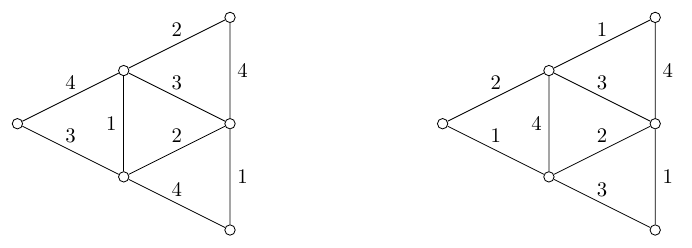}
\caption{A series-parallel and outerplanar graph $G$ with $\chi'_{trans}(G)=4$.}
\end{figure}

\begin{example} Consider the  outerplanar  (and series-parallell) graph $G$ in Figure 2 with
	two different proper $4$-edge colorings $f$ 
	(to the left) and $g$ 
	(to the right).
	The same argument as in the preceding example shows that
$g$ cannot be  obtained from $f$ even by a sequence of
	transformations each of which uses  at most three color classes, that is, $\chi'_{trans}(H)=4$.
\end{example}

A graph $G$ is called {\it line perfect} if the line graph of $G$ is a perfect graph. 

\begin{corollary}
If  $G$ is a line perfect graph, then $\chi'_{trans}(G)\leq 4$.
\end {corollary}

\begin{proof}
Let $G$ be a line perfect graph.
Then (see \cite{maffray})  
 every block of $G$ is either a bipartite graph, a complete graph $K_4$, or a graph $K_{1,1,n}$, for some $n\geq  1$.
Since $K_4$ and $K_{1,1,n}$ are $3$-degenerate graphs in  the set $\cal A$$(3)$, the corollary follows from 
 \cref{main:th} under $q=3$.
\end{proof}

\smallskip

\begin{remark} 
There are infinite families of graphs $H$ with $\chi'_{trans}(H)=4$, where every block is a bipartite or series-parallell graph.

One such family   of graphs can be constructed as follows.
Consider  the set of $\cal C$ of all square $(n,m)$-grids ($n\geq 3, m\geq 3$). 
For a grid $F\in \cal C$,  take a  copy $G_F$  of the graph $G$ in Figure 2, disjoint from $F$, 
and identify a vertex of degree $2$
in $F$ with  a vertex of degree $2$ in  $G_F$. This  vertex is a cutvertex of the obtained graph $H=H(F,G)$. 
Clearly, $\Delta(H)=4=\chi'(H)$, $H\in \cal A$$(3)$ and $H$ has two blocks, $F$ and $G_F$. 
Since $H$ contains $G_F$ as a block,
 it must satisfy $\chi'_{trans}(H) \geq 4$.
On the other hand, it follows from Corollary \ref{th:B}, that $\chi'_{trans}(H) \leq 4$. Thus, $\chi'_{trans}(H)=4$.

Next,  we  define  an infinite sequence  of disjoint graphs $H_0, H_1,...$ with $\Delta(H_i)=4=\chi'_{trans}(H_i)$, 
for $i=0,1,...$, where every block is a  series-parallel graph. Put $H_0$ be the graph $G$ in Figure 2.
Assume that we already defined  disjoint graphs $H_0,...,H_k$, $k\geq 0$. Take a  copy $G_k$  of the graph $G$ in Figure 2, 
disjoint from $\cup_{i=0}^kH_i$, 
and identify a vertex of degree $2$
in $H_k$ with  a vertex of degree $2$ in  $G_k$. Clearly, every block of the obtained graph $H_{k+1}$ is a series-parallel graph.
As in the preceding paragraph, 
we can show that  $\chi'_{trans}(H_{k+1})=4$.
\end{remark}

Now we consider planar bipartite multigraphs. The following result  for  all bipartite multigraphs  was obtained by 
Asratian and Mirumian \cite{asratian91} (a shorter proof appears  in  \cite{asratian09}):

\begin{theorem} [\cite{asratian91}]
\label{note}
 If  $G$ is a bipartite multigraph, then  any proper $\Delta(G)$--coloring  of $G$ can be transformed into any other 
 proper $\Delta(G)$--coloring of $G$
by a series of transformations such that each of the intermediate colorings is a proper $\Delta(G)$--coloring of $G$ and involves
at most $3$ color classes of the previous coloring. 
 \end{theorem}

\begin{figure}
\includegraphics[width=10cm]{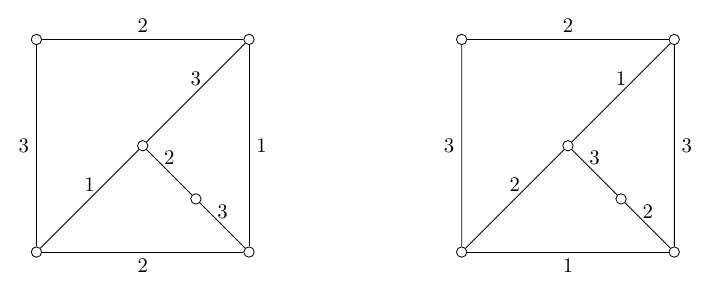}
\caption{A planar bipartite graph $G$ with $\chi'_{trans}(G)=3.$}
\end{figure}

In our terminology  \cref{note} means that $\chi'_{trans}(G)\leq 3$ for every bipartite multigraph $G$. 
In fact, this bound is sharp even for planar
bipartite graphs. Consider, for example,  the graph in Figure 3 with two proper 3-colorings $f$ (on the left) and $g$ (on the right). 
Every Kempe change in $f$ gives the same partition of the edge set of $G$. Therefore $f$ and $g$ are not 
Kempe-equivalent, that is,  $\chi'_{trans}(G)=3$. 

A  better result was found for 3-regular planar bipartite graphs.
 
\begin{theorem}[Belcastro and Haas~\cite{belcastro14}]
\label{belcastro}
Let $G$ be a 3-regular planar  bipartite multigraph. Then all proper $3$-colorings of $G$ are Kempe-equivalent.
\end{theorem}

Using   \cref{note} and \cref{belcastro} we obtain here the following generalization of the result of Belcastro and Haas:

\begin{corollary}
\label{asratian19}
Let $G$ be a $t$-regular planar  bipartite multigraph, $t\geq 2$. Then $\chi'_{trans}(G)=2$, 
that is, all proper $t$-colorings of $G$ are 
Kempe-equivalent.
\end{corollary}


\begin{proof}
By \cref{note}, there is a sequence of proper $t$-colorings  
$f=f_0,f_1,\dots, f_r$ so that $f_r=g$
 and  $f_{i+1}$ and
$f_i$ are $3$-equivalent
 for each $i=0,1,\dots,r-1$.
If $f_{i+1}$ differs from $f_i$ by two color classes, then $f_i$ and $f_{i+1}$ are Kempe-equivalent.
 Suppose that  $f_{i+1}$ differs from $f_i$ only by three color classes, say the classes $M(f_i,\alpha), M(f_i,\beta)$ and 
 $M(f_i,\gamma)$.
Let $D=M(f_i,\alpha)\cup M(f_i,\beta)\cup M(f_i,\gamma)$. Then
$D=M(f_{i+1},\alpha)\cup M(f_{i+1},\beta)\cup M(f_{i+1},\gamma)$ too. Consider the 
planar $3$-regular bipartite
subgraph $H$ of $G$ induced by the edges in $D$, and
denote by $f_i'$ and $f_{i+1}'$ the $3$-colorings of $H$  induced  by $f_i$ and $f_{i+1}$, respectively.
Then, by \cref{belcastro}, $f_i'$ and $f_{i+1}'$ are Kempe-equivalent, which implies that 
$f_i$ and $f_{i+1}$ are also Kempe-equivalent.
Thus $f_i$ and $f_{i+1}$ are Kempe-equivalent, for each $i=0,\dots,r-1$. Therefore $f$ and $g$ 
are Kempe-equivalent.
\end{proof}

\section{Concluding remarks}

In this section  we discuss Problem \ref{prob:trans}  formulated in the introduction
as well as its analog for vertex  colorings.
First we prove that it suffices to consider Problem \ref{prob:trans} for Class 1 graphs only.

\begin{proposition} 
\label{prop:Class1}
The answer to Problem \ref{prob:trans} is positive if and only if such a constant $C$ exists for 
all  Class 1 graphs. 
\end{proposition}

\begin{proof}
If $\chi'_{trans}(G)\leq C$ for every  graph or multigraph $G$, then $\chi'_{trans}(G)\leq C$ for every  Class 1 graph $G$.

Conversely, assume  that $\chi'_{trans}(H)\leq C$ for every  Class 1 graph $H$, and let $G$ be either a multigraph
or a Class 2 graph.

We will associate with $G$ a Class 1 graph $F_G$ as follows:
Let $\chi'(G)=t$ and $E(G)=\{e_1,\dots,e_n\}$ where $e_i=u_iv_i$, $i=1,\dots,n$.
For each edge $e_i$  we define a set of new vertices 
$$V_{e_i}=\{u_{e_i}(0), u_{e_i}(1),\dots,u_{e_i}(t-1), v_{e_i}(0), v_{e_i}(1),\dots,v_{e_i}(t-1)\}$$
 such that 
$V_{e_i}\cap V(G)=\emptyset$ and $V_{e_i}\cap V_{e_j}=\emptyset$ for $1\leq i<j\leq n$. 

Furthermore, with every edge $e_i$ we associate a bipartite graph $G_{e_i}$ with vertex set $V(G_{e_i})=V_{e_i}\cup \{u_i,v_i\}$ 
and edge set
$$E(G_{e_i})=\{u_iv_{e_i}(0), v_iu_{e_i}(0)\}\cup (\{u_{e_i}(s)v_{e_i}(l) : 0\leq s, l\leq t-1\})\setminus \{u_{e_i}(0)v_{e_i}(0)\}.$$
Now we define the graph $F_G$ as the union of bipartite graphs $G_{e_1},\dots,G_{e_n}$, that is,
$F_G=\cup_{i=1}^nG_{e_i}$.
Clearly, $F_G$ is a    graph with maximum degree $t$. 
With every  proper $t$-coloring $f$ of $G$ we associate a  proper $t$-coloring $f'$ of $F_G$ as follows:
For each $i=1,\dots,n$,  we properly color the bipartite graph $G_{e_i}$ with colors $1,2,\dots,t$ such that
the edges $u_iv_{e_i}(0)$ and $v_iu_{e_i}(0)$ receive the color $f(e_i)$. This is possible because 
the graph obtained from $G_{e_i}$ by deleting the vertices $u_i$ and $v_i$, and adding the edge $u_{e_i}(0)v_{e_i}(0)$
is a complete bipartite graph $K_{t,t}$.
 The obtained $t$-coloring $f'$ of $F_G$ is proper and $\chi'(F_G)=t$,  that is,   $F_G$ is a Class 1 graph.
 Then, by our assumption, $\chi'_{trans}(F_G)\leq C$.  
 
 It is not difficult to verify that  for any proper $t$-coloring $f'$ of $F_G$  we have  $f'(u_iv_{e_i}(0))=f'(v_iu_{e_i}(0))$, 
 for $i=1,\dots,n$. 
 Then every proper $t$-coloring of $F_G$ induces a  proper $t$-coloring $f$ of the multigraph $G$
 by taking $f(e_i)=f'(u_iv_{e_i}(0))$, for $i=1,\dots,n$. This implies that $\chi'_{trans}(G)\leq \chi'_{trans}(F_G)\leq C$.
\end{proof}

In general, Problem \ref{prob:trans} remains open.  The next result shows that   the answer to a similar problem for vertex colorings is negative.
Let us first recall some definitions.  A vertex set $U\subseteq V(G)$ is {\em independent} in a graph $G$,
 if no two vertices of $U$ are adjacent in $G$. A {\em proper vertex $t$-coloring} of $G$
  is a partition of $V(G)$ in $k$ independent sets $U_1, \dots, U_t$, called color classes. 
The minimum number $t$ for which there exists a  proper vertex $t$-coloring of $G$ is called 
the {\em chromatic number} of  $G$ and is denoted by $\chi(G)$.

  \begin{proposition}
\label{th:phenomen}
For every integer 
$n\geq 3$  there exist an infinite class $\cal G$$(n)$ of regular graphs 
with chromatic number $n$  such that for every $G$ in $\cal G$$(n)$ any two proper 
$n$-colorings  of $G$
 can  be transformed to each other  only by a "global  transformation" involving all $n$ color classes.
\end{proposition}

\begin{proof}
For any pair of integers $n\geq 3$ and $p\geq 1$ define the graph $H(p,n)$ as follows:
its vertex set is $V_1\cup\dotsb\cup V_{2n}$,
where $V_1,\dotsc,V_{2n}$ are pairwise disjoint sets of cardinality $p$,
and two vertices of $H(p,n)$ are adjacent if and only if
they both belong to $V_1\cup V_{2n}$ or to $V_i\cup V_{i+1}$
for some $i\in\{1,2,\dotsc,2n-1\}$.
Clearly, $H(p,n)$ is a $(3p-1)$-regular graph with $2pn$ vertices where
 the maximum size of a clique  is $2p$.

Put   $\cal G$$(n)=\{\overline{H(p,n)} : p=1,2,\dots\}$ where $\overline{H(p,n)}$ denote the complement of the graph $H(p,n)$.
Clearly, $\overline{H(p,n)}$ is a $(2np-3p)$-regular  graph 
with  $2pn$ vertices where the maximum number of independent vertices is $2p$. 
This implies that the chromatic number of $\overline{H(p,n)}$ is at least $n$. 
Furthermore, 
the vertices of $\overline{H(p,n)}$ can be properly colored with $n$ colors if  
the vertex set of $\overline{H(p,n)}$
  can be partitioned into $n$ disjoint independent sets of size $2p$. 
This is equivalent to partitioning the set of vertices  of $H(p,n)$ into $n$ disjoint 
cliques of size $2p$.

 It is not difficult to see that $H(p,n)$ has  only two distinct partitions of the set of its vertices   into $n$ cliques of size $2p$.
 One of the partitions is $\{C_1,\dots,C_n\}$ where $C_i=V_{2i-1}\cup V_{2i}$, for $i=1,\dots,n$,
 and the other is $\{R_1,\dots,R_n\}$ where $R_1=V_{2n}\cup V_1$ and $R_{i}=V_{2i-2}\cup V_{2i-1}$, for $i=2,\dots,n$.
 
This implies that the chromatic number of  $\overline{H(p,n)}$ is $n$ and $\overline{H(p,n)}$ has only two distinct proper $n$-colorings. 
One of them, say $f$,  is obtained when the color $i$ is received by the vertices in the set $C_i=V_{2i-1}\cup V_{2i}$,  for $i=1,\dots,n$,
and the other, say $g$,  is obtained when  the vertices in $R_1=V_1\cup V_{2n}$ receive the color $1$  and  the vertices  in $R_i=V_{2i-2}\cup V_{2i-1}$,  receive the color $i$, for $i=2,\dots,n$.

It follows that the subgraph induced by any  $s$ color classes 
from $\{C_1,\dots,C_n\}$  of $f$ 
is distinct from the subgraph induced by any $s$ color classes of $g$ from
$\{R_1,\dots,R_n\}$, for every $s \leq n-1$. 
Thus in order to transform $f$ into $g$ we have to change all  color classes of $f$.
\end{proof}

\end{document}